\documentclass[oneside,11pt]{amsart}
\usepackage{amssymb,amscd,amsmath,latexsym,graphicx,}
\usepackage[mathcal]{euscript}
 \usepackage{hyperref}
 
\newcommand{\ind}{\operatorname{Ind}}

\newcommand{\bx}[3]
{\put(#2,#3)
{
\begin{picture}(10,10)(0,0)
\put(0,10){\line(1,0){10}}
\put(10,0){\line(0,1){10}}
\put(0,0){\line(1,0){10}}
\put(0,0){\line(0,1){10}}
\put(2.5,1.3){#1}
\end{picture}
}}
\newtheorem{theorem}{Theorem}[section]
\newtheorem{definition}[theorem]{Definition}
\newtheorem{proposition}[theorem]{Proposition}
\newtheorem*{notation}{Notation}

\numberwithin{equation}{section}

\begin{document}

\title[Decomposing induced characters]{Decomposing induced characters of the centralizer of an n-cycle in the symmetric group on 2n elements}
\author{\sc Joseph J. Ricci}
\address{Department of Mathematics\\ University at Buffalo, SUNY \\
244 Mathematics Building\\Buffalo, NY~14260, USA}
\email{jjricci@math.buffalo.edu}

\date{January 3 2012}
 
 
\begin{abstract}
We give explicit multiplicities and formulas for multiplicities of characters appearing in the decomposition of the induced character $\ind_{C_{S_{2n}}(\sigma)}^{S_{2n}}1_C$, where $\sigma$ is an $n$-cycle, $C_{S_{2n}}(\sigma)$ is the centralizer of $\sigma$ in $S_{2n}$, and $1_C$ denotes the trivial character on $C_{S_{2n}}(\sigma)$.
\end{abstract}
\maketitle
\section{Introduction}
Throughout this paper we work only over the complex numbers, dealing with $\mathbb{C}S_n$ characters. Let $\sigma \in S_n$. In a natural way, $\sigma \in S_{2n}$ as well. Let $C:=C_{S_{2n}}(\sigma)$ be the centralizer of $\sigma$ in $S_{2n}$. Let $\psi$ be any linear character of $C$. In \cite{Hemmer}, Hemmer showed that for $m \geq n$ the induced character $\ind_{C}^{S_m} \psi$ becomes representation stable for $m = 2n$. Therefore, these induced characters arise naturally when studying braid group cohomology. (For more on representation stability and braid group cohomology, see \cite{Church}.) It was proposed that in general the decomposition of the induced character $\ind_{C}^{S_{2n}} \psi$ into irreducible characters of $S_{2n}$ was an open problem.  \\
 \indent However, the case when $\sigma = (1 \ 2 \ \cdots n)$ was studied in \cite{Jollenback} and \cite{Weyman}. In this case, $C_{S_n} (\sigma)= \langle \sigma \rangle$. Then the linear characters of $C$ are precisely the irreducible characters, which are indexed by the numbers $k=0,\ 1, \ldots, \ n-1$ and take $\sigma$ to $e^\frac{{2\pi i k}}{n}$. It was shown that for an irreducible character $\chi^\lambda$ of $S_{n}$, the multiplicity of $\chi^\lambda$ in the decomposition of $\ind_{\langle \sigma \rangle}^{S_n} \psi_k$ is equal to the number of standard Young tableaux of shape $\lambda$ with major index congruent to $k$ mod $n$. Once this is computed, one can use the Littlewood-Richardson rule or the branching rule to induce the resulting characters up to $S_{2n}$. So, in theory, the decomposition of $\ind_{C}^{S_{2n}}\psi_k$ is known, however no explicit formula is available in general. \\
 \indent In this paper we will deal with the case when $\sigma$ is an $n$-cycle of $S_n$ and $\psi_k = 1_C$ (i.e. $k=0$), the trivial character. We present a partial result toward an explicit formula as well as a formula for the multiplicities of certain irreducible $\mathbb{C}S_{2n}$ characters appearing in the decomposition.\\
\indent This paper was submitted to the University at Buffalo as the author's undergraduate honors thesis. The work was inspired by Dr. David Hemmer and the question posed in the closing of \cite{Hemmer}. Hemmer also served as the author's advisor and oversaw the progress on the paper. \\ \indent 

\section{Preliminaries}
\subsection{Partitions and Young diagrams}\label{S:parts}
\begin{definition} $\lambda = (\lambda_1,\ldots, \lambda_r)$ is a {\bf{\emph{partition}}} of $n$, written $\lambda \vdash n$, if $\lambda_i \geq \lambda_{i+1} \geq 0$, each $\lambda_i \in \mathbb{Z}$, and $\lambda_1 + \cdots + \lambda_r = n$. We say each $\lambda_i$ is a {\bf{\emph{part}}} of $\lambda$.
\end{definition}

\begin{definition} Let $\lambda = (\lambda_1,\ldots, \lambda_r) \vdash n$. Then the {\bf{\emph{Young diagram}}}, $\left[\lambda\right]$, of $\lambda$ is the set
\begin{equation}
\left[ \lambda \right] = \{(i,j) \in \mathbb{N} \times \mathbb{N} \ | \ j \leq \lambda_i \}. \notag
\end{equation}
We say each $(i,j) \in \left[ \lambda \right]$ is a {\bf{\emph{node}}} of $\left[ \lambda \right]$.
\end{definition}

If $\lambda \vdash n$, we represent $\left[ \lambda \right]$ by an array of boxes. As an example, consider the partition $\lambda = (5,3,2,2,1) \vdash 13$. Then we visualize  $\left[ \lambda \right]$ as:

\begin{figure}[ht]
\begin{center}\setlength{\unitlength}{0.015in}
\begin{picture}(50,70)(0,0)
\bx{}{0}{0}
\bx{}{0}{10}
\bx{}{10}{10}
\bx{}{0}{20}
\bx{}{10}{20}
\bx{}{0}{30}
\bx{}{10}{30}
\bx{}{20}{30}
\bx{}{0}{40}
\bx{}{10}{40}
\bx{}{20}{40}
\bx{}{30}{40}
\bx{}{40}{40}
\end{picture}
\end{center}
\end{figure}

\noindent where the upper left box is defined to be the ordered pair $(1,1)$, the upper right is $(1,5)$, the lower left is $(5,1)$, just like the entries of a matrix. \\
\indent We will often drop the bracket notation and use $\lambda$ and $\left[\lambda\right]$ interchangeably, though it will be clear by context to which we are referring. If $\lambda_i$ is a part of $\lambda \vdash n$, then $\lambda / \lambda_i$ is the partition of $n - \lambda_i$ formed by deleting $\lambda_i$ from $\lambda$. So $(5,3,2,2,1) / \lambda_2 = (5,2,2,1)$. If $b=(i,j)$ is a node in the Young diagram of $\lambda$, we will write $b \in \lambda$. Suppose $\mu = (3,2,1)$. Returning to our previous example, it is easy to see that each node $b \in \mu $ is also a node of $\lambda$. We will denote this in the obvious way, $\mu \subseteq \lambda$. With this idea in mind, we make a definition.

\begin{definition} Let $\lambda$ and $\mu$ be partitions such that $\mu \subseteq \lambda$. Then the {\bf{\emph{skew diagram}}} $\lambda / \mu$ is the set of nodes $$\xi = \lambda / \mu = \{b \in \lambda | b \not \in \mu\}.$$
\end{definition}

\noindent In the case of our example, the skew diagram $\lambda / \mu$ in this case would be:

\begin{figure}[!h]
\begin{center}\setlength{\unitlength}{0.015in}
\begin{picture}(50,70)(0,0)
\bx{}{0}{0}
\bx{}{0}{10}
\bx{}{10}{10}
\bx{}{10}{20}
\bx{}{20}{30}
\bx{}{30}{40}
\bx{}{40}{40}
\end{picture}
\end{center}
\end{figure}

\noindent One important aspect of Young diagrams that will be of great important in this paper are rim hooks.

\begin{definition} For a skew diagram $\xi$, we say the unique node $(i_0,j_0)$ such that $i_0 \leq i$ and $j_0 \geq j$ for all $(i,j) \in \xi$ is the {\bf{\emph{top node}}} of $\xi$. 
\end{definition}

\begin{definition} A {\bf{\emph{rim hook}}} is a skew diagram $\xi$ such that  if $(i,j)$ is not the top node of $\xi$ then either $(i-1,j) \in \xi$ or $(i,j+1) \in \xi$, but not both.
\end{definition}

\noindent We will say a rim $k$-hook or simply a $k$-hook is a rim hook consisting of $k$ nodes. We will say that a partition $\lambda$ \emph{has} a $k$-hook if it is possible to remove a $k$-hook from $\lambda$ and have the resulting diagram be the Young diagram of some partition $\lambda^\prime$. To each rim hook $\xi$ is assigned the leg length of $\xi$.

\begin{definition} Let $\xi$ be a rim hook. The {\bf{\emph{leg length of $\xi$}}}, denoted $ll(\xi)$, is 
$$ll(\xi) = {\text{(the number of rows in $\xi$)}} -1.$$
\end{definition}

\noindent Once again returning to our example where $\lambda = (5,3,2,2,1)$, we see that $\lambda$ has three rim 4-hooks:

\begin{figure}[ht] 
\centering 
\begin{minipage}{1.2in}{\includegraphics{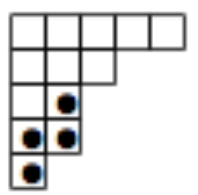}}
\end{minipage}
\qquad
\begin{minipage}{1.2in}{\includegraphics{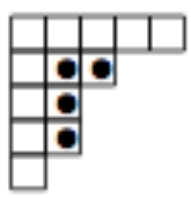}}
\end{minipage}
\qquad
\begin{minipage}{1.2in}{\includegraphics{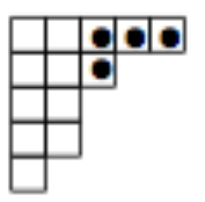}}
\end{minipage}
\end{figure}

\noindent In the first and third cases, the 4-hooks have leg length 2, while in the second case the 4-hook has leg length 1. One can also see that $\lambda$ does not have any rim 5-hooks, since it is not possible to remove a 5-hook from $\lambda$ and have the resulting diagram be the Young diagram of a partition.

\subsection{Character theory of the symmetric group}

The basics of representation and character theory will be assumed, and can be found in \cite{James}. It is well known (\cite[2.3.4, 2.4.4]{Sagan}) that there is a one-to-one correspondence between the set of partitions of $n$ and the set of irreducible characters of $S_n$. For example, $\chi^{(n)}$ corresponds to the trivial character, $\chi^{(n-1,1)}$ corresponds to the number of fixed points minus one, and $\chi^{(1^n)}$ corresponds to the sign character. Also, the conjugacy classes of $S_n$ have a natural correspondence to the partitions of $n$. If $\tau \in S_n$ is of cycle type $\lambda$, $\lambda \vdash n$, then we will denote the conjugacy class of $\tau$ by $K_\lambda$. Let $\lambda, \mu \vdash n$. Suppose one wants to evaluate the character $\chi^{\lambda}$ on the conjugacy class $K_{\mu}$, which we will denote $\chi^{\lambda}_{\mu}$. Then we have the following theorem which allows one to recursively compute $\chi^\lambda_\mu$:

\begin{theorem}[Murnaghan-Nakayama Rule] \cite[4.10.2]{Sagan} \label{mnrule}
Let $\mu=(\mu_1 , \ldots, \mu_s), \lambda \vdash n$. Then $$\chi^{\lambda}_{\mu} = \displaystyle\sum_\xi (-1)^{ll(\xi)} \chi^{\lambda / \xi}_{\mu / \mu_1}$$where the sum is taken over all rim hooks $\xi$ of $\lambda$ containing $\mu_1$ nodes.
\end{theorem}

\noindent Now, in a natural way, one can think of $S_{n-1}$ as a subgroup of $S_{n}$. Suppose $\chi^{\lambda}$ is the character of $S_n$ corresponding to $\lambda$ and $\chi^{\mu}$ is the character of $S_{n-1}$ corresponding to $\mu$. Then one can easily compute the restricted character $\chi^{\lambda} \downarrow_{S_{n-1}}$ and the induced character $\ind_{S_{n-1}}^{S_n} \chi^{\mu}$ using the Branching Rule. 

\begin{definition}
Let $\lambda \vdash n$. We say an {\bf{\emph{inner corner}}} of $\left[ \lambda \right]$ is a node $(i,j) \in \left[ \lambda \right]$ such that $\left[ \lambda \right] - \{(i,j)\}$ is the Young diagram of some partition of $n-1$. We denote any such partition by $\lambda^-$. We say an {\bf{\emph{outer corner}}} is a node $(i,j) \not \in \left[ \lambda \right]$ such that $\left[ \lambda \right] \cup \{(i,j)\}$ is the Young diagram of some partition of $n+1$. We denote any such partition by $\lambda^+$. 
\end{definition}

\begin{theorem}[Branching Rule] \cite[2.8.3]{Sagan} \label{branchingrule}
Let $\mu \vdash n-1, \lambda \vdash n$. Then $$\chi^{\lambda} \downarrow_{S_{n-1}} = \displaystyle\sum_{\lambda^-} \chi^{\lambda^-}$$
and
$$\ind_{S_{n-1}}^{S_n} \chi^{\mu} = \displaystyle\sum_{\mu^+} \chi^{\mu^+}.$$ 
\end{theorem}

As an example, suppose $\lambda = (3,3,2)$ and $\mu = (5,2)$. Using Theorem \ref{branchingrule} we calculate
\begin{align}
\chi^{(3,3,2)} \downarrow_{S_7} &= \chi^{(3,2,2)} + \chi^{(3,3,1)} \notag \\
\ind_{S_7}^{S_8} \chi^{(5,2)} &= \chi^{(6,2)} + \chi^{(5,3)} + \chi^{(5,2,1)}. \notag
 \end{align} 

\section{The decomposition of $\phi$} 
\subsection{Some preliminary results} Recall in the introduction of this paper we defined $C:=C_{S_{2n}}(\sigma)$, with $\sigma = (1 \ 2 \cdots n)$. One can compute \cite[4.3]{Dummit} $C \cong \langle \sigma \rangle \times S_n$. Keeping this in mind we note the following notation:

\begin{notation} For $\tau \in C$, we will write $\tau = (\sigma^k, \pi)$ for $k \in \mathbb{Z}$ and $\pi \in S_n$. 
\end{notation}

\noindent Also, if $\lambda = (\lambda_{1}, \ldots, \lambda_{r}) \vdash n$ then $(n, \lambda):=(n, \lambda_1, \ldots, \lambda_r) \vdash 2n$ and $(\lambda, 1^n):=(\lambda_1, \ldots, \lambda_r, 1^n) \vdash 2n$. 

\begin{notation}When evaluating any character $\chi$ on the conjugacy class of $S_{2n}$ corresponding to $(n,\lambda)$ or $(\lambda, 1^n)$, we will write $\chi_{(n,\lambda)}$ and $\chi_{(\lambda, 1^n)}$ respectively. 
\end{notation}

\noindent For the remainder of this paper, we will write $\phi = \ind_C^{S_{2n}} 1$. 

\begin{proposition}\label{P:1}
Let $n \geq 1$. Let $\chi^{(2n)}$ be the irreducible character of $S_{2n}$ corresponding to the partition $(2n)$. Then
\begin{align}
\langle \phi, \chi^{(2n)} \rangle_{S_{2n}} &= 1. \notag
\end{align}
\end{proposition}

\begin{proof}

Using Frobenius reciprocity, we have 
\begin{equation}
\langle \phi, \chi^{(2n)} \rangle_{S_{2n}} = \langle 1_C, \chi^{(2n)}\downarrow_C \rangle_C. \notag
\end{equation}
But since $\chi^{(2n)}$ is the trivial character, $\chi^{(2n)} \downarrow C = 1_C$, so we have:
\begin{equation}
\langle \phi, \chi^{(2n)} \rangle_{S_{2n}} = 1. \notag
\end{equation}
\end{proof}

\begin{proposition}\label{P:2}
Let $n \geq 2$. Let $\chi^{(2n-1,1)}$ be the irreducible character of $S_{2n}$ corresponding to $(2n-1,1)$. Then
\begin{equation}
\langle \phi, \chi^{(2n-1,1)} \rangle_{S_{2n}} = 1. \notag
\end{equation}
\end{proposition}

\begin{proof}

First note that this character records the number points fixed by a permutation and subtracts 1. Using Frobenius reciprocity, we expand the inner product as follows:
\begin{equation}
\langle \phi, \chi^{(2n-1,1)} \rangle_{S_{2n}} = \langle 1_C, \chi^{(2n-1,1)} \downarrow_{C} \rangle_C = \frac{1}{n \ n!}\displaystyle\sum_{\tau \in C} \chi^{(2n-1,1)} (\tau). \label{1}
\end{equation}
By previous remarks made at the beginning of this section, the last term in (\ref{1}) becomes:
\begin{equation}
\frac{1}{n \ n!}\displaystyle\sum_{k=0}^{n-1} \displaystyle\sum_{\pi \in S_n} \chi^{(2n-1,1)} ((\sigma^k,\pi)). \notag
\end{equation}
When $k=0$, $(\sigma^k,\pi) = (1,\pi)$ and $(1,\pi)$ fixes $n + \chi^{(n-1,1)}(\pi) +1$ points. When $k \not = 0$, $(\sigma^k, \pi)$ fixes $\chi^{(n-1,1)}(\pi)+1$ points, giving:
\begin{align}
\frac{1}{n \ n!}\displaystyle\sum_{k=0}^{n-1} \displaystyle\sum_{\pi \in S_n} \chi^{(2n-1,1)} ((\sigma^k,\pi)) &= \frac{1}{n \ n!}\left[\displaystyle\sum_{\pi \in S_n} \left(n+ \chi^{(n-1,1)}(\pi)\right) + (n-1)\displaystyle\sum_{\pi \in S_n} \chi^{(n-1,1)}(\pi)\right] \notag \\
&= \frac{1}{n \ n!}\left[\displaystyle\sum_{\pi \in S_n} n + \displaystyle\sum_{\pi \in S_n} \chi^{(n-1,1)}(\pi) +(n-1) \displaystyle\sum_{\pi \in S_n} \chi^{(n-1,1)}(\pi)\right] \notag \\
&= \frac{1}{n \ n!}\left[n \ n! + n \ n! \langle \chi^{(n)} , \chi^{(n-1,1)} \rangle_{S_n}\right]. \label{2}
\end{align}
But since both $\chi^{(n)}$ and $\chi^{(n-1,1)}$ are irreducible, their inner product is 0. So (\ref{2}) becomes:
\begin{equation}
\frac{1}{n \ n!} n \ n! = 1. \notag
\end{equation}
\end{proof}

\begin{proposition}\label{P:3}
Let $n \geq 2$. Let $\chi^{(n,n)}$ be the irreducible character of $S_{2n}$ corresponding to $(n,n)$. Then
\begin{align}
\langle \phi, \chi^{(n,n)} \rangle_{S_{2n}} &= 1. \notag
\end{align}
\end{proposition}

\begin{proof}
Throughout, let $d_k=gcd(n,k)$. Using Frobenius reciprocity, we write:
\begin{equation}
\langle \phi, \chi^{(n,n)} \rangle_{S_{2n}} = \langle 1_C, \chi^{(n,n)} \downarrow_{C} \rangle_{C} = \frac{1}{n \ n!}\displaystyle\sum_{k=0}^{n-1} \displaystyle\sum_{\pi \in S_n} \chi^{(n,n)} ((\sigma^k,\pi)). \notag
\end{equation}
We break the sum up into three pieces: one for $k=0$, one for $d_k=1$ (of which there are $\varphi(n)$ such $k$, where $\varphi$ denotes Euler's totient function) and one for $d_k \not = 1$:
\begin{align}
\langle \phi, \chi^{(n,n)} \rangle_{S_{2n}} &=
\frac{1}{n \ n!}\left[ \displaystyle\sum_{\pi \in S_n} \chi^{(n,n)}((1,\pi)) + \varphi(n)\displaystyle\sum_{\pi \in S_n} \chi^{(n,n)}((\sigma,\pi)) + \displaystyle\sum_{\begin{subarray}{1} 
1<k<n\\
d_k \not = 1
\end{subarray}} \displaystyle\sum_{\pi \in S_n} \chi^{(n,n)}((\sigma^k, \pi))\right]. \notag \\
\intertext{In order to use Theorem \ref{mnrule}, we sum over all partitions of $n$ and rewrite the sum as:}
\langle \phi, \chi^{(n,n)} \rangle_{S_{2n}} &= \frac{1}{n \ n!}\left[n! \langle \chi^{(n)}, \chi^{(n,n)}\downarrow_{S_n} \rangle_{S_n} + \varphi(n) \displaystyle\sum_{\lambda \vdash n} \chi^{(n,n)}_{(n,\lambda)} |K_{\lambda}|+  \displaystyle\sum_{\begin{subarray}{1} 
1<k<n \\ 
d_k \not = 1 
\end{subarray}}\displaystyle\sum_{\lambda \vdash n} \chi^{(n,n)}_{((\frac{n}{d_k})^{d_k},\lambda)} |K_{\lambda}|\right]. \label{3}
\end{align}
By Theorem \ref{branchingrule}, we write:
\begin{equation}
\chi^{(n,n)}\downarrow_{S_n} = \chi^{(n)} + \displaystyle\sum_
{\begin{subarray}{1}
\lambda \vdash n \\ 
\lambda \not = (n) 
\end{subarray}}
a_\lambda \chi^\lambda \notag
\end{equation}
where $a_\lambda \in \{0, 1, 2,\ldots\}$. Then by linearity, we have 
\begin{equation}
\langle \chi^{(n)}, \chi^{(n,n)}\downarrow_{S_n} \rangle_{S_n} = \langle \chi^{(n)}, \chi^{(n)} \rangle_{S_n} + \displaystyle\sum_
{\begin{subarray}{1}
\lambda \vdash n \\ 
\lambda \not = (n) 
\end{subarray}}
a_\lambda \langle \chi^{(n)}, \chi^{\lambda} \rangle_{S_n} = 
\langle \chi^{(n)}, \chi^{(n)} \rangle_{S_n} = 1\label{4}
\end{equation}
since all the $\chi^{\lambda}$ are irreducible. Using Theorem \ref{mnrule}, 
\begin{equation}
\chi^{(n,n)}_{(n,\lambda)} = \chi^{(n)}_\lambda - \chi^{(n-1,1)}_\lambda \notag
\end{equation}
so that 
\begin{align}
\displaystyle\sum_{\lambda \vdash n} \chi^{(n,n)}_{(n,\lambda)} |K_{\lambda}| 
&= \displaystyle\sum_{\lambda \vdash n} \left( \chi^{(n)}_\lambda - \chi^{(n-1,1)}_\lambda \right) |K_{\lambda}| \notag \\
&= \displaystyle\sum_{\lambda \vdash n} \chi^{(n)}_\lambda |K_{\lambda}| - \displaystyle\sum_{\lambda \vdash n} \chi^{(n-1,1)}_\lambda |K_{\lambda}|  \notag \\
&= n! \langle \chi^{(n)}, \chi^{(n)} \rangle_{S_n} - n! \langle \chi^{(n)}, \chi^{(n-1,1)} \rangle_{S_n} \notag \\
&= n!. \label{5}
\end{align}

\noindent Now let $d_k \not = 1$, for some $k$. Again with Theorem \ref{mnrule}, we write 
\begin{equation}
\chi^{(n,n)}_{((\frac{n}{d_k})^{d_k},\lambda)} = \chi^{(n)}_\lambda + \displaystyle\sum_
{\begin{subarray}{1}
\mu \vdash n \\ 
\mu \not = (n) 
\end{subarray}}
c_\mu \chi^{\mu}_{\lambda} \notag
\end{equation} 
where $c_\lambda \in \mathbb{Z}$.  Then 
\begin{align}
\displaystyle\sum_{\lambda \vdash n} \chi^{(n,n)}_{((\frac{n}{d_k})^{d_k},\lambda)} |K_{\lambda}| 
&= \displaystyle\sum_{\lambda \vdash n} \chi^{(n)}_{\lambda}|K_{\lambda}| + \displaystyle\sum_{\lambda \vdash n}\displaystyle\sum_
{\begin{subarray}{1}
\mu \vdash n \\ 
\mu \not = (n) 
\end{subarray}}
c_\mu \chi^{\mu}_{\lambda} |K_{\lambda}| \notag \\
&= n! \langle \chi^{(n)}, \chi^{(n)} \rangle_{S_n} + \displaystyle\sum_
{\begin{subarray}{1}
\mu \vdash n \\ 
\mu \not = (n) 
\end{subarray}} 
n! c_\mu \langle \chi^{(n)}, \chi^{\mu} \rangle_{S_n}  \notag \\
&=n!. \label{6}
\end{align} 

\noindent We note that there are $n-\varphi(n)-1$ numbers $k$ strictly between 1 an $n$ so that $d_k \not = 1$, so substituting (\ref{4}), (\ref{5}), and(\ref{6}) into (\ref{3}) we have:
\begin{equation}
\langle \phi, \chi^{(n,n)} \rangle_{S_{2n}} = \frac{1}{n \ n!} \left[n! + \varphi(n)n! + (n-\varphi(n) -1)n! \right] = \frac{1}{n \ n!} n \ n! = 1. \notag
\end{equation}
\end{proof}

In the case of $n=2$ it turns out that Propositions \ref{P:1}, \ref{P:2}, and \ref{P:3} give a full decomposition. That is:
\begin{equation}
\ind_{C_{S_4}((12))}^{S_4} 1_C = \chi^{(4)} + \chi^{(3,1)} + \chi^{(2,2)}. \notag
\end{equation}

\noindent We notice that our first three results all showed that there are certain irreducible characters appearing in the decomposition of $\phi$ that have constant or stable multiplicities, independent of $n$. Our next result shows that this is not the case for all constituents, but a closed form formula for the multiplicity is known in some cases.

\begin{proposition}Let $n \geq 2$. Let $\chi^{(2n-2,2)}$ be the irreducible character of $S_{2n}$ corresponding to $(2n-2,2)$. Then 
\begin{equation}
\langle \phi, \chi^{(2n-2,2)} \rangle_{S_{2n}} =  
\begin{cases}
\frac{n}{2} 	&\text{if $n$ is even}\\
\frac{n-1}{2} &\text{if $n$ is odd}
\end{cases} \notag
\end{equation}
\end{proposition}

\begin{proof}
Throughout, $d_k=gcd(n,k)$. Using Frobenius reciprocity we write:
\begin{equation}
\langle \phi, \chi^{(2n-2,2)} \rangle_{S_{2n}} = \langle 1_C, \chi^{(2n-2,2)} \downarrow_C \rangle_C = \frac{1}{n \ n!}\displaystyle\sum_{k=0}^{n-1}\displaystyle\sum_{\pi \in S_n} \chi^{(2n-2,2)} ((\sigma^k, \pi)). \notag
\end{equation}
If $n=2$, we are done, by Proposition \ref{P:3}. Throughout the rest of the proof we assume $n \geq 3$.  As in the proof of Proposition \ref{P:3}, we break the sum into three pieces:
\begin{align}
\langle \phi, \chi^{(2n-2,2)} \rangle_{S_{2n}}
&= \frac{1}{n \ n!}\left[ \displaystyle\sum_{\pi \in S_n} \chi^{(2n-2,2)}((1,\pi)) + \varphi(n)\displaystyle\sum_{\pi \in S_n} \chi^{(2n-2,2)}((\sigma,\pi)) + \displaystyle\sum_
{\begin{subarray}{1} 
1<k<n \\ 
d_k \not = 1 
\end{subarray}}\displaystyle\sum_{\pi \in S_n} \chi^{(2n-2,2)}((\sigma^k, \pi))\right] \notag \\
&= \frac{1}{n \ n!} \left[ n! \langle \chi^{(n)} , \chi^{(2n-2,2)} \rangle_{S_n} + \varphi(n) \displaystyle\sum_{\lambda \vdash n} \chi^{(2n-2,2)}_{(n, \lambda)} |K_\lambda| + \displaystyle\sum_{\begin{subarray}{1} 
1<k<n \\
d_k \not = 1
\end{subarray}}\displaystyle\sum_{\lambda \vdash n} \chi^{(2n-2,2)}_{((\frac{n}{d_k})^{d_k},\lambda)} |K_\lambda| \right]. \label{7}
\end{align}

\noindent From Theorem \ref{branchingrule}, we have

\begin{equation}
\langle \chi^{(n)}, \chi^{(2n-2,2)} \rangle_{S_n} =\binom{n}{2}. \label{8}
\end{equation}

\noindent Using Theorem \ref{mnrule} we write
\begin{equation} 
\chi^{(2n-2,2)}_{(n, \lambda)} = \chi^{(n-2,2)}_{\lambda} \notag
\end{equation}
so that 
\begin{equation}
\displaystyle\sum_{\lambda \vdash n} \chi^{(2n-2,2)}_{(n, \lambda)} |K_\lambda| = \displaystyle\sum_{\lambda \vdash n} \chi^{(n-2,2)}_{\lambda} |K_\lambda| = n! \langle \chi^{(n)} , \chi^{(n-2,2)} \rangle_{S_n} = 0. \label{9}
\end{equation}

\noindent When $n$ is even, $\frac{n}{2}$ divides $n$. Then $d_{\frac{n}{2}} = \frac{n}{2}$. We can then remove the 2-hook from bottom row of $(2n-2,2)$, and then successively remove $\frac{n}{2} - 1$ hooks of length 2 from the top row of $(2n-2,2)$. There are $\binom{\frac{n}{2}}{1}=\frac{n}{2}$ ways to do this. We combine this with Theorem \ref{mnrule} to see that 
\begin{equation}
\displaystyle\sum_{\begin{subarray}{1} 
1<k<n \\ 
d_k \not = 1
\end{subarray}} 
\chi^{(2n-2,2)}_{((\frac{n}{d_k})^{d_k},\lambda)} = \frac{n}{2}\chi^{(n)} + \displaystyle\sum_{\begin{subarray}{1}
\mu \vdash n \\
\mu \not = (n) 
\end{subarray}}
a_\mu \chi^{\mu}_{\lambda} \notag
\end{equation}
with $a_\mu \in \mathbb{Z}$. Then 
\begin{align}
\displaystyle\sum_{\begin{subarray}{1} 
1<k<n \\ 
d_k \not = 1
\end{subarray}} 
\displaystyle\sum_{\lambda \vdash n} \chi^{(2n-2,2)}_{((\frac{n}{d_k})^{d_k},\lambda)} |K_\lambda| &= \displaystyle\sum_{\lambda \vdash n} \frac{n}{2}\chi^{(n)} |K_\lambda| + \displaystyle\sum_{\lambda \vdash n} \displaystyle\sum_{\begin{subarray}{1}
\mu \vdash n \\
\mu \not = (n) 
\end{subarray}}
a_\mu \chi^{\mu}_{\lambda} |K_\lambda|  \notag \\
&=  \frac{n}{2} n! \langle \chi^{(n)} , \chi^{(n)} \rangle_{S_n} + \displaystyle\sum_{\begin{subarray}{1}
\mu \vdash n \\
\mu \not = (n) 
\end{subarray}} 
a_\mu \langle \chi^{(n)}, \chi^\mu \rangle_{S_n}  \notag \\
&= \frac{n}{2} n!. \label{10}
\end{align}

\noindent So in the case when $n$ is even, substituting (\ref{8}), (\ref{9}), and (\ref{10}) into (\ref{7}), we have 
\begin{equation}
\langle \phi, \chi^{(2n-2,2)} \rangle_{S_{2n}}
= \frac{1}{n \ n!}\left[\binom{n}{2}n! + \frac{n}{2}n! \right]
= \frac{1}{n} \left[ \binom{n}{2} + \frac{n}{2} \right] 
= \frac{1}{n} \left[ \frac{n(n-1)}{2} + \frac{n}{2} \right]
= \frac{n-1}{2} + \frac{1}{2}
= \frac{n}{2} \notag
\end{equation}

\noindent as desired. Now, when $n$ is odd, $2$ does not divide $n$. Then $\frac{n}{2}$ is not an integer and thusly does not divide $n$. As a result, we cannot remove the hook of length 2 from the bottom row of $(2n-2,2)$. So when we apply Theorem \ref{mnrule}, the trivial character does not appear in the decomposition and we have
\begin{equation}
\displaystyle\sum_{\begin{subarray}{1} 
1<k<n \\ 
d_k \not = 1 
\end{subarray}} 
\chi^{(2n-2,2)}_{((\frac{n}{d_k})^{d_k},\lambda)} = \displaystyle\sum_{\begin{subarray}{1}
\mu \vdash n \\
\mu \not = (n) 
\end{subarray}}
c_\mu \chi^{\mu}_{\lambda} \notag
\end{equation}
with $c_\mu \in \mathbb{Z}$. Then 
\begin{equation}
\displaystyle\sum_{\begin{subarray}{1}
1<k<n \\ 
d_k \not = 1
\end{subarray}} 
\displaystyle\sum_{\lambda \vdash n} \chi^{(2n-2,2)}_{((\frac{n}{d_k})^{d_k},\lambda)} |K_\lambda| = \displaystyle\sum_{\lambda \vdash n} \displaystyle\sum_{\begin{subarray}{1}
\mu \vdash n \\
\mu \not = (n) 
\end{subarray}}
c_\mu \chi^{\mu}_{\lambda} |K_\lambda| =  \displaystyle\sum_{\begin{subarray}{1}
\mu \vdash n \\
\mu \not = (n) 
\end{subarray}} 
n! c_\mu \langle \chi^{(n)}, \chi^\mu \rangle_{S_n} = 0. \label{11}
\end{equation}
So then substituting (\ref{8}), (\ref{9}), and (\ref{11}) into (\ref{7}), we have
\begin{equation}
\langle \phi, \chi^{(2n-2,2)} \rangle_{S_{2n}}
= \frac{1}{n \ n!}\binom{n}{2}n!
= \frac{1}{n} \binom{n}{2} 
= \frac{1}{n} \frac{n(n-1)}{2}
= \frac{n-1}{2} \notag
\end{equation}
giving the result.
\end{proof}

\subsection{A theorem for the partitions $(2n-k,k)$}
We now present a theorem that generalizes the previous propositions and gives a formula for the multiplicities of a number of the irreducible characters of $S_{2n}$ appearing in the decomposition of $\phi$. 

\begin{theorem}
Let $n \geq 2k$. Let $\chi^{(2n-k,k)}$ be the irreducible character of $S_{2n}$ corresponding to $(2n-k,k)$. For $1<h<n$, let $d_h=gcd(n,h)$, and $l_k=\frac{k d_h}{n}$. Then
\begin{equation}
\langle \phi, \chi^{(2n-k,k)} \rangle_{S_{2n}} =\frac{1}{n} \bigg[ \binom{n}{k} + \displaystyle\sum_{\begin{subarray}{1}
1<h<n \\
d_h \not = 1 \\
\frac{n}{d_h} \big| k
\end{subarray}}
\binom{d_h}{l_k}\bigg]. \notag
\end{equation} 
\end{theorem}

\begin{proof}
With Frobenius reciprocity, we write
\begin{equation}
\langle \phi, \chi^{(2n-k,k)} \rangle_{S_{2n}} = \langle 1_C, \chi^{(2n-k,k)}\downarrow_{C} \rangle_{C} = \frac{1}{n \ n!} \displaystyle\sum_{j=0}^{n-1} \sum_{\pi \in S_n} \chi^{(2n-k,k)}((\sigma^j,\pi)). \notag
\end{equation}
As usual, we split the sum into three pieces:
\begin{align}
\langle \phi, \chi^{(2n-k,k)}\downarrow_{S_{2n}} \rangle_{S_{2n}} 
&= \frac{1}{n \ n!}\left[ \displaystyle\sum_{\pi \in S_n} \chi^{(2n-k,k)} ((1,\pi)) + \varphi(n)\displaystyle\sum_{\pi \in S_n} \chi^{(2n-k,k)}((\sigma,\pi)) + \displaystyle\sum_
{\begin{subarray}{1}
1<h<n \\
d_h \not = 1
\end{subarray}}
\sum_{\pi \in S_n} \chi^{(2n-k,k)} ((\sigma^h, \pi)) \right] \notag \\
&= \frac{1}{n \ n!}\left[ n!\langle \chi^{(n)}, \chi^{(2n-k,k)}\downarrow_{S_n} \rangle_{S_n} + \varphi(n)\displaystyle\sum_{\lambda \vdash n} \chi^{(2n-k,k)}_{(n,\lambda)} |K_\lambda|  + \displaystyle\sum_
{\begin{subarray}{1}
1<h<n \\
d_h \not = 1
\end{subarray}}
\sum_{\lambda \vdash n} \chi^{(2n-k,k)}_{((\frac{n}{d_h})^{d_h}, \lambda)} |K_\lambda| \right]. \label{12}
\end{align}

\noindent Since $n \geq 2k$, we can remove the $k$ blocks from the bottom row of $(2n-k,k)$ and remove $n-k$ blocks from the top row of $(2n-k,k)$, which leaves $n$ blocks remaining. We can do this removal in $\binom{n}{k}$ ways, so with $\ref{branchingrule}$, we have
\begin{equation}
\langle \chi^{(n)}, \chi^{(2n-k,k)}\downarrow_{S_n} \rangle_{S_n} = \binom{n}{k}. \label{13}
\end{equation}
Theorem \ref{mnrule} gives
\begin{equation}
\chi^{(2n-k,k)}_{(n,\lambda)} = \chi^{(n-k,k)}_{\lambda} \notag
\end{equation}

\noindent since $n \geq 2k$. Then
\begin{equation}
\displaystyle\sum_{\lambda \vdash n} \chi^{(2n-k,k)}_{(n,\lambda)} |K_\lambda| = \displaystyle\sum_{\lambda \vdash n} \chi^{(n-k,k)}_{\lambda} |K_\lambda| = n! \langle \chi^{(n)} , \chi^{(n-k,k)}  \rangle_{S_n} =0.  \label{14}
\end{equation}

\noindent Now suppose there is some $h$ so that $d_h \not = 1$. Then $\sigma^h$ is a product of $d_h$ $\frac{n}{d_h}$-cycles. If $\pi$ is of cycle type $\lambda$ then
\begin{equation}
\chi^{(2n-k,k)}((\sigma^h,\pi))=\chi^{(2n-k,k)}_{((\frac{n}{d_h})^{d_h},\lambda)}. \label{15}
\end{equation}
By Theorem \ref{mnrule}, in order for $\chi^{(n)}$ to have non-zero multiplicity in the decomposition of $\chi^{(2n-k,k)}_{((\frac{n}{d_h})^{d_h},\lambda)}$, we have to be able to remove the $k$-hook from the bottom row of $(2n-k,k)$. So if $\frac{n}{d_h}$ does not divide $k$ then this is not possible. Then in this case
\begin{equation}
\chi^{(2n-k,k)}_{((\frac{n}{d_h})^{d_h},\lambda)} = \displaystyle\sum_{\begin{subarray}{1}
\ \ \mu \vdash n \\
\ \ \mu \not = (n) 
\end{subarray}}
a_\mu \chi^{\mu}_{\lambda} \label{16}
\end{equation}

\noindent with $a_\mu \in \mathbb{Z}$. Now, suppose that for some $h$, $d_h \not = 1$ and furthermore that $\frac{n}{d_h}$ divides $k$. Then we can successively remove the $l_k$ hooks of length $\frac{n}{d_h}$ from the bottom row of $(2n-k,k)$ and remove the $d_h-l_k$ hooks of length $\frac{n}{d_h}$ from the top row of $(2n-k,k)$, which will result in $\chi^{(n)}$ having positive multiplicity in the aforementioned decomposition. In fact, a simple counting argument via Theorem \ref{branchingrule} shows the exact multiplicity will be $\binom{d_h}{l_k}$. Then in this case
\begin{equation}
\chi^{(2n-k,k)}_{((\frac{n}{d_h})^{d_h},\lambda)} = \binom{d_h}{l_k}\chi^{(n)}_{\lambda} + \displaystyle\sum_{\begin{subarray}{1}
\ \ \mu \vdash n \\
\ \ \mu \not = (n) 
\end{subarray}}
c_\mu \chi^{\mu}_{\lambda} \label{17}
\end{equation}
with $c_\mu \in \mathbb{Z}$. Then (\ref{16}) and (\ref{17}) give

\begin{align}
\displaystyle\sum_
{\begin{subarray}{1}
1<h<n \\
d_h \not = 1
\end{subarray}}
\sum_{\lambda \vdash n} \chi^{(2n-k,k)}_{((\frac{n}{d_h})^{d_h}, \lambda)} |K_\lambda| 
&=  \displaystyle\sum_
{\begin{subarray}{1}
1<h<n \\
d_h \not = 1 \\
\frac{n}{d_h} \nmid k 
\end{subarray}}
\displaystyle\sum_{\lambda \vdash n} \displaystyle\sum_{\begin{subarray}{1}
\ \ \mu \vdash n \\
\ \ \mu \not = (n) 
\end{subarray}}
a_\mu \chi^{\mu}_{\lambda} |K_\lambda| + \displaystyle\sum_
{\begin{subarray}{1}
1<h<n \\
d_h \not = 1 \\
\frac{n}{d_h} \big| k
\end{subarray}}
\displaystyle\sum_{\lambda \vdash n}  \binom{d_h}{l_k}\chi^{(n)}_{\lambda} |K_\lambda| + \displaystyle\sum_
{\begin{subarray}{1}
1<h<n \\
d_h \not = 1 \\
\frac{n}{d_h} \big| k
\end{subarray}}
\displaystyle\sum_{\lambda \vdash n} \displaystyle\sum_{\begin{subarray}{1}
\ \ \mu \vdash n \\
\ \ \mu \not = (n) 
\end{subarray}}
c_\mu \chi^{\mu}_{\lambda} |K_\lambda| \notag \\
&= \displaystyle\sum_
{\begin{subarray}{1}
1<h<n \\
d_h \not = 1 \\
\frac{n}{d_h} \big| k
\end{subarray}}
\displaystyle\sum_{\lambda \vdash n}  \binom{d_h}{l_k}\chi^{(n)}_{\lambda} |K_\lambda| \notag \\
&= \displaystyle\sum_
{\begin{subarray}{1}
1<h<n \\
d_h \not = 1 \\
\frac{n}{d_h} \big| k
\end{subarray}} \binom{d_h}{l_k} n!. \label{18}
\end{align} 

\noindent Substituting (\ref{13}), (\ref{14}), (\ref{18}) into (\ref{12}) we have
\begin{equation}
\langle \phi, \chi^{(2n-k,k)} \rangle_{S_{2n}} = \frac{1}{n \ n!} \bigg[ \binom{n}{k}n! + \displaystyle\sum_
{\begin{subarray}{1}
1<h<n \\
d_h \not = 1\\
\frac{n}{d_h} \big| k
\end{subarray}} \binom{d_h}{l_k} n! \bigg] = \frac{1}{n} \bigg[ \binom{n}{k} + \displaystyle\sum_
{\begin{subarray}{1}
1<h<n \\
d_h \not = 1\\
\frac{n}{d_h} \big| k
\end{subarray}} \binom{d_h}{l_k} \bigg]
\end{equation}
as claimed.
\end{proof}

\section{Future problems} The preceding work is only the beginning of a large selection of problems to be worked out. It is possible that there are more stable multiplicities (independent of $n$) in this decomposition. Also, the multiplicities and formulas found here only cover a small number of partitions and therefore characters. One may find that \emph{all} characters have a stable or closed form formula for their multiplicities. Note that in this paper we only discuss the trivial character of $C$, and much can be learned from studying the decomposition of the non-trivial characters of $C$ when induced up to $S_{2n}$, which arise in braid group cohomology. It may be possible to learn more by first decomposing the character $\ind_{C}^{S_n} \psi$, studying this character, and then inducing the resulting constituents up to $S_{2n}$.

\end{document}